\documentclass[12pt,oneside]{article}
\usepackage[top=3cm, bottom=3cm, left=3cm, right=3cm]{geometry}
\usepackage{hyperref}
\hypersetup{
   colorlinks = true,
   linkcolor = blue,
   anchorcolor = blue,
   citecolor = blue,
   filecolor = blue,
    urlcolor = blue
    }
\usepackage{tikz}

\usetikzlibrary{decorations.markings,arrows}
\usetikzlibrary{shapes,calc}
\usepackage{amsmath}
\usepackage{amssymb}
\usepackage{amsthm}
\allowdisplaybreaks 
\usepackage{enumerate}
\usepackage{mathtools}
\usepackage{enumerate}
\usepackage{makecell} 

\newtheorem{theorem}{Theorem}

\newtheorem{lemma}{Lemma} 
 
\numberwithin{equation}{section}

\newcolumntype{C}[1]{>{\centering\arraybackslash$}m{#1}<{$}}

\newcolumntype{R}[1]{>{\raggedleft\arraybackslash$}m{#1}<{$}}

\newcolumntype{L}[1]{>{\raggedright\arraybackslash$}m{#1}<{$}}

\begin{document}

\title{Sweet division problems: from chocolate bars to honeycomb strips and back}

\author{Tomislav Do\v{s}li\'c \\
    Faculty of Civil Engineering, University of Zagreb\\
    Croatia, Zagreb 10000\\
  \text{tomislav.doslic@grad.unizg.hr}\\
 \and
Luka Podrug\\
    Faculty of Civil Engineering, University of Zagreb\\
    Croatia, Zagreb 10000\\
  \text{luka.podrug@grad.unizg.hr}}
\date{}
\maketitle

\pagestyle{plain}

\begin{abstract} \noindent
We consider two division problems on narrow strips of square and hexagonal
lattices. In both cases we compute the bivariate enumerating sequences and
the corresponding generating functions, which allowed us to determine 
the asymptotic behavior of the total number of such subdivisions and the 
expected number of parts. For the square lattice we extend results of two
recent references by establishing polynomiality of enumerating sequences
forming columns and diagonals of the triangular enumerating sequence. 
In the hexagonal case, we find a number of new combinatorial interpretations
of the Fibonacci numbers and find combinatorial proofs of some Fibonacci
related identities. We also show how both cases could be treated via the
transfer matrix method and discuss some directions for future research.
\end{abstract}

\noindent

\section{Introduction}
\noindent
It has been a long standing problem of great practical importance to count the
ways of dividing a collection of entities into smaller sets according to a
given set of rules. If the entities are considered to be indivisible, and
we only care about their number, the natural framework for modeling such
situations is the theory of integer partitions and compositions, depending
on further properties of the considered entities. If, on the other hand,
we are interested in relationships between the entities, such as, e.g.,
their adjacency patterns or their relative positions, we must resort to
more complex models such as graphs and geometric figures.

In this paper we look at finite portions of the square and hexagonal
regular lattices, and count ways of dividing narrow strips in such lattices
into a given number of pieces while preserving integrity of individual
squares or hexagons. The considered portions of square and hexagonal lattices
remind us on chocolate bars and honeycomb slabs, respectively, hence the title.
We start by revisiting some partial results for narrow
strips in the square lattice available in the literature and present complete
solution to the problem. In particular, we derive the recurrences satisfied
by the sequences enumerating the divisions of a $2 \times n$ strip into
$k$ pieces. From them, we compute the bivariate generating function whose
univariate specialization yields the recurrence for the overall number of
divisions. In that way we recover the results of Knopfmacher obtained in
the context of compositions of ladder graphs \cite{knopf}. We refine those
results by investigating behavior of columns in the enumerating triangle.
We establish
convolution-type recurrences for all columns, going thus beyond partial
results of references \cite{brown,durham,wagon}. Then we apply the same approach
to narrow strips of hexagons, again deriving the recurrences and computing
the bivariate generating function. Along the way we find a new combinatorial
representation of odd-indexed Fibonacci numbers and provide a new combinatorial proof for one well-known identity for Fibonacci numbers. Then we show how the results for
honeycomb strips can be obtained by using transfer matrices. Finally, we also
derive transfer matrices for the chocolate bars we started from.

The paper is concluded by some remarks on the strong and weak points of
employed methods and with some indications of possible further directions.

\section{Definitions and preliminary results}

Let $n$ be a non-negative integer. We consider a $2 \times n$ rectangular
strip consisting of $2n$ squares arranged in 2 rows and $n$ columns such as the one shown in Fig. \ref{fig:rec_grid n}. In the rest of the paper we will
often refer to such strips as to chocolate bars of length $n$. We consider 
divisions of such structures into a given number of pieces obtained by
cutting along the edges of basic squares. More precisely, we would like to
find the number of all possible divisions of such a bar of a given length, 
and also the number of such divisions into a given number of parts $k$.
Clearly, $1 \leq k \leq 2n$ are the only meaningful values of $k$. Let $r_k(n)$
denote the number of divisions of $2\times n$ rectangular strip into exactly
$k$ pieces and $r(n)$ the total number of divisions. From definition we have
that $r_k(n)=0$ for $ k < 0$ and for $k>2n$. The initial values are
$r_1(1)=r_2(1)=1$, and $r_1(2)=1$, $r_2(2)=6$, $r_3(2)=4$ and $r_1(4)=1$.

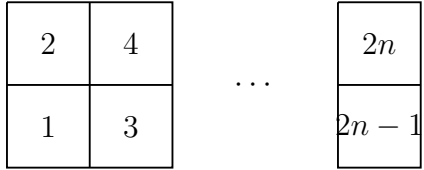
\begin{figure}[ht]
\centering

\begin{tikzpicture}[scale=0.55]
\coordinate (M1) at (0,2);
\coordinate (M2) at (2,2);
\coordinate (M3) at (4,2);   
\coordinate (M4) at (6,2);
\coordinate (M5) at (8,2);
\coordinate (M6) at (10,2); 

\coordinate (N1) at (0,0);
\coordinate (N2) at (2,0);
\coordinate (N3) at (4,0);   
\coordinate (N4) at (6,0);
\coordinate (N5) at (8,0);
\coordinate (N6) at (10,0); 

\coordinate (P1) at (0,-2);
\coordinate (P2) at (2,-2);
\coordinate (P3) at (4,-2);   
\coordinate (P4) at (6,-2);
\coordinate (P5) at (8,-2);
\coordinate (P6) at (10,-2);

\coordinate (A) at ({sqrt(3)},1);
\coordinate (B) at ({sqrt(3)/2},0.5);
\coordinate (C) at ({sqrt(3)/2},-0.5);   
\coordinate (D) at ({sqrt(3)},-1);
\coordinate (E) at ({3*sqrt(3)/2},-0.5);
\coordinate (F) at ({3*sqrt(3)/2},0.5); 

\coordinate (A1) at ({2*sqrt(3)},1);
\coordinate (B1) at ({3*sqrt(3)/2},0.5);
\coordinate (C1) at ({3*sqrt(3)/2},-0.5);   
\coordinate (D1) at ({2*sqrt(3)},-1);
\coordinate (E1) at ({5*sqrt(3)/2},-0.5);
\coordinate (F1) at ({5*sqrt(3)/2},0.5); 

\coordinate (A2) at ({3*sqrt(3)},1);
\coordinate (B2) at ({5*sqrt(3)/2},0.5);
\coordinate (C2) at ({5*sqrt(3)/2},-0.5);   
\coordinate (D2) at ({3*sqrt(3)},-1);
\coordinate (E2) at ({7*sqrt(3)/2},-0.5);
\coordinate (F2) at ({7*sqrt(3)/2},0.5);

\coordinate (A3) at ({4*sqrt(3)},1);
\coordinate (B3) at ({7*sqrt(3)/2},0.5);
\coordinate (C3) at ({7*sqrt(3)/2},-0.5);   
\coordinate (D3) at ({4*sqrt(3)},-1);
\coordinate (E3) at ({9*sqrt(3)/2},-0.5);4
\coordinate (F3) at ({9*sqrt(3)/2},0.5);

\coordinate (A4) at ({5*sqrt(3)},1);
\coordinate (B4) at ({9*sqrt(3)/2},0.5);
\coordinate (C4) at ({9*sqrt(3)/2},-0.5);   
\coordinate (D4) at ({5*sqrt(3)},-1);
\coordinate (E4) at ({11*sqrt(3)/2},-0.5);4
\coordinate (F4) at ({11*sqrt(3)/2},0.5);

\coordinate (A5) at ({6*sqrt(3)},1);
\coordinate (B5) at ({11*sqrt(3)/2},0.5);
\coordinate (C5) at ({11*sqrt(3)/2},-0.5);   
\coordinate (D5) at ({6*sqrt(3)},-1);
\coordinate (E5) at ({13*sqrt(3)/2},-0.5);4
\coordinate (F5) at ({13*sqrt(3)/2},0.5);

\coordinate (G) at ({sqrt(3)/2},-0.5);
\coordinate (H) at ({0},-1);
\coordinate (I) at ({0},-2);   
\coordinate (J) at ({sqrt(3)/2},-2.5);
\coordinate (K) at ({sqrt(3)},-2);
\coordinate (L) at ({sqrt(3)},-1);

\coordinate (G1) at ({3*sqrt(3)/2},-0.5);
\coordinate (H1) at ({{sqrt(3)}},-1);
\coordinate (I1) at ({{sqrt(3)}},-2);   
\coordinate (J1) at ({3*sqrt(3)/2},-2.5);
\coordinate (K1) at ({2*sqrt(3)},-2);
\coordinate (L1) at ({2*sqrt(3)},-1);

\coordinate (G2) at ({5*sqrt(3)/2},-0.5);
\coordinate (H2) at ({{2*sqrt(3)}},-1);
\coordinate (I2) at ({{2*sqrt(3)}},-2);   
\coordinate (J2) at ({5*sqrt(3)/2},-2.5);
\coordinate (K2) at ({3*sqrt(3)},-2);
\coordinate (L2) at ({3*sqrt(3)},-1);

\coordinate (G3) at ({7*sqrt(3)/2},-0.5);
\coordinate (H3) at ({{3*sqrt(3)}},-1);
\coordinate (I3) at ({{3*sqrt(3)}},-2);   
\coordinate (J3) at ({7*sqrt(3)/2},-2.5);
\coordinate (K3) at ({4*sqrt(3)},-2);
\coordinate (L3) at ({4*sqrt(3)},-1);

\coordinate (G4) at ({9*sqrt(3)/2},-0.5);
\coordinate (H4) at ({{4*sqrt(3)}},-1);
\coordinate (I4) at ({{4*sqrt(3)}},-2);   
\coordinate (J4) at ({9*sqrt(3)/2},-2.5);
\coordinate (K4) at ({5*sqrt(3)},-2);
\coordinate (L4) at ({5*sqrt(3)},-1);

\coordinate (G5) at ({11*sqrt(3)/2},-0.5);
\coordinate (H5) at ({{5*sqrt(3)}},-1);
\coordinate (I5) at ({{5*sqrt(3)}},-2);   
\coordinate (J5) at ({11*sqrt(3)/2},-2.5);
\coordinate (K5) at ({6*sqrt(3)},-2);
\coordinate (L5) at ({6*sqrt(3)},-1);

\draw [line width=0.25mm] (M1)--(M3)--(P3)--(P1)--(M1); 
\draw[line width=0.25mm] (M2)--(P2);
\draw[line width=0.25mm] (N1)--(N3);
\draw [line width=0.25mm] (M5)--(M6)--(P6)--(P5)--(M5); 
\draw [line width=0.25mm] (N5)--(N6); 

\node[] (p) at (6,0) {$\cdots$};
\node[] (p) at (1,-1) {$1$};
\node[] (p) at (1,1) {$2$};
\node[] (p) at (3,-1) {$3$};
\node[] (p) at (3,1) {$4$};
\node[] (p) at (9,1) {$2n$};
\node[] (p) at (9,-1) {$2n-1$};

\end{tikzpicture} 
\caption{Rectangular strip containing $2n$ squares.} \label{fig:rec_grid n}
\end{figure} 

In a recent paper, Brown \cite{brown} studied such divisions and obtained a
system of recursive relations that we include below as Theorem \ref{tm:Brown}.
In order to state the Brown's results, we need one auxiliary term, more
specifically, the number of divisions of a $2\times n$ rectangular strip
into $k$ parts such that the squares in the last column belong to different
parts. We denote that number by $q_k(n)$ and show one such division in Figure \ref{fig:rec_grid examp} as an example.

\begin{figure}[ht]
\centering
\begin{tikzpicture}[scale=0.55]
\draw [line width=0.25mm, fill=green] (M2)--(M1)--(P1)--(P4)--(M4)--(M3)--(N3)--(N2)--(M2); 
\draw[line width=0.25mm,fill=yellow] (M2)--(M3)--(N3)--(N2)--(M2);
\draw [line width=0.25mm,fill=orange] (M4)--(M6)--(N6)--(N5)--(P5)--(P4)--(M4); 
\draw [line width=0.25mm, fill=blue] (N5)--(N6)--(P6)--(P5)--(N5);
\draw[dashed,opacity=0.4] (N1)--(N2) (N3)--(N5)  (N2)--(P2) (N3)--(P3) (M5)--(N5); 
 
\end{tikzpicture} 
\caption{One division of $2\times 5$ rectangular strip into $4$ parts with squares in the last column being in different parts. Total number of such divisions is denoted by $q_4(5)$. } \label{fig:rec_grid examp}
\end{figure}
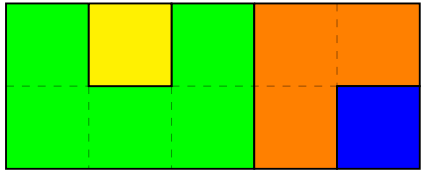 

\begin{theorem}[Brown] 
The number of divisions of $2\times n$ strip into $k$ parts satisfies following system of equations:
\begin{align*}
r_{k}(n+1)&=r_{k}(n)+3r_{k-1}(n)+r_{k-2}(n)+2q_{k}(n)\\
q_{k}(n+1)&=2r_{k-1}(n)+r_{k-2}(n)+q_{k}(n).
\end{align*}\label{tm:Brown} \end{theorem}
It is an easy exercise to eliminate $q_k(n)$ from the system of equations in
Theorem  \ref{tm:Brown} and to obtain recursive relations for $r_{k}(n)$,
\begin{equation} \label{eq:r_k(n)}
r_{k}(n+1)=r_{k-2}(n)+3r_{k-1}(n)+2r_{k}(n)+r_{k-2}(n-1)+r_{k-1}(n-1)-r_{k}(n-1),\end{equation}
and for the overall number of such divisions,
\begin{equation} \label{eq:r(n)} r(n+1)=6r(n)+r(n-1).
\end{equation}

These recurrences will serve as the starting point of our Section 3, where
the Brown's results will be extended and refined by establishing recurrences
in $n$ for a fixed $k$ and by computing the expected values of $k$ in a
random division of a $2 \times n$ chocolate bar. \\

Next we consider a hexagonal strip composed of $n$ regular hexagons as shown
in Figure \ref{fig:grid1}. Throughout the paper, such  hexagonal strips
will be also referred to as honeycomb strips. As in the rectangular case, 
we are interested in counting the number of all possible divisions of this
strip that contains exactly $k$ pieces. Note that here $n$ does not denote
the number of columns but the total number of hexagons in the strip. Hence
we consider $k$ to be any integer between (and including) $1$ and $n$.
Again, only divisions along the edges of hexagons are considered. The
hexagons are added in the order as it is shown in Figure \ref{fig:grid1}.  

\begin{figure}[ht]
\centering

\begin{tikzpicture}[scale=0.55]

\draw [line width=0.25mm, fill=yellow] (A1) -- (B1) -- (C1) -- (H1) -- (I1) -- (J1) --(K1)--(J2)--(K2)--(J3)--(K3)--(L3)--(E3)--(D4)--(E4)--(F4)--(A4)--(B4)--(A3)--(B3)--(G3)--(H3)--(G2)--(F1)--(A1); 

\draw [line width=0.25mm, fill=orange] (A5)--(B5)--(C5)--(H5)--(G4)--(H4)--(I4)--(J4)--(K4)--(J5)--(K5)--(L5)--(E5)--(F5)--(A5); 
\draw [line width=0.25mm, fill=cyan] (A2) -- (B2) -- (C2) -- (D2) -- (E2) -- (F2) --(A2); 

\draw [line width=0.25mm, fill=green] (A) -- (B) -- (C) -- (H) -- (I) -- (J) --(K)--(L)--(E)--(F)--(A); 

\draw[dashed,opacity=0.4] (D) -- (C) (C1) -- (D1)--(E1) (D1)--(K1) (D2)--(K2) (C3)--(D3) (E3)--(F3) (K4)--(L4) (C5)--(D5); 

\node[] (p) at ({sqrt(3)},0) {2};
\node[] (p) at ({2*sqrt(3)},0) {4};
\node[] (p) at ({3*sqrt(3)},0) {6};
\node[] (p) at ({4*sqrt(3)},0) {8};
\node[] (p) at ({5*sqrt(3)},0) {10};
\node[] (p) at ({6*sqrt(3)},0) {12};

\node[] (p) at ({sqrt(3)/2},-1.5) {1};
\node[] (p) at ({3*sqrt(3)/2},-1.5) {3};
\node[] (p) at ({5*sqrt(3)/2},-1.5) {5};
\node[] (p) at ({7*sqrt(3)/2},-1.5) {7};
\node[] (p) at ({9*sqrt(3)/2},-1.5) {9};
\node[] (p) at ({11*sqrt(3)/2},-1.5) {11};
\end{tikzpicture}  
\caption{Honeycomb strip with 12 hexagons divided into 4 pieces.} \label{fig:grid1}
\end{figure}
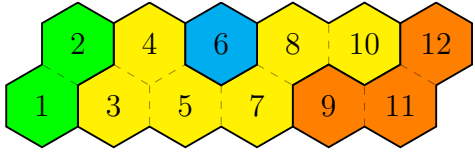 

Let $D_k(n)$ denote the set of all possible divisions of the honeycomb strip
with $n$ hexagons into $k$ pieces and $d_k(n)=|D_k(n)|$ the number of elements
of the set $D_k(n)$. Now we can state some simple cases: $d_1(n)=1$, for every
non-negative integer $n$, since there is only one way to obtain one part,
and $d_n(n)=1$, since there is only one way to obtain $n$ parts, that is to
let each hexagon form its own part. Furthermore, $d_k(n)=0$ for $k < 1$ and
for $k>n$. It is convenient to set $d_1(0)=1$. As an example, we list all
possible divisions of the strip containing $4$ hexagons as the first
non-trivial case.
\begin{align*}
d_1(4)=1 \phantom{-}& \left\lbrace 1234\right\rbrace\\
d_2(4)=6 \phantom{-}& \left\lbrace 1,  234\right\rbrace,\left\lbrace 2,  134\right\rbrace,\left\lbrace 3, 124\right\rbrace,\left\lbrace 4,  123\right\rbrace,\left\lbrace 12,  34\right\rbrace,\left\lbrace 13, 24\right\rbrace\\
d_3(4)=5 \phantom{-}& \left\lbrace 12,3,4\right\rbrace,\left\lbrace 13,2,4\right\rbrace,\left\lbrace 23,1,4\right\rbrace,\left\lbrace 24,1,3\right\rbrace,\left\lbrace 34,1,2\right\rbrace\\
d_4(4)=1 \phantom{-}& \left\lbrace 1,2,3,4\right\rbrace
\end{align*}
Note that the division $\left\lbrace 14,23\right\rbrace$ is not included,
since hexagons $1$ and $4$ are not adjacent, as shown in
Figure \ref{fig:grid n=4}, thus cannot form a part.
\begin{figure}[ht]
\centering

\begin{tikzpicture}[scale=0.55]

\draw [line width=0.25mm,fill=green] (A)--(B)--(C)--(H)--(I)--(J)--(K)--(J1)--(K1)--(L1)--(E1)--(F1)--(A1)--(B1)--(A); 
\draw[dashed,opacity=0.4] (F)--(E)--(D)--(C);
\draw[dashed,opacity=0.4] (C1) -- (D1); 
\draw[dashed,opacity=0.4] (D1) -- (E1);  
\draw[dashed,opacity=0.4] (K) -- (L);  
\node[] (p) at ({sqrt(3)},0) {2};
\node[] (p) at ({2*sqrt(3)},0) {4};
\node[] (p) at ({sqrt(3)/2},-1.5) {1};
\node[] (p) at ({3*sqrt(3)/2},-1.5) {3};
\end{tikzpicture} 

\begin{tikzpicture}[scale=0.55]

\draw [line width=0.25mm,fill=yellow] (G)--(H)--(I)--(J)--(K)--(L)--(G) ; 

\draw [line width=0.25mm,fill=green] (A)--(B)--(C)--(D)--(I1)--(J1)--(K1)--(L1)--(E1)--(F1)--(A1)--(B1)--(A); 

\draw[dashed,opacity=0.4] (D)--(E) -- (F) (C1)--(D1); 
\node[] (p) at ({sqrt(3)},0) {2};
\node[] (p) at ({2*sqrt(3)},0) {4};
\node[] (p) at ({sqrt(3)/2},-1.5) {1};
\node[] (p) at ({3*sqrt(3)/2},-1.5) {3};
\end{tikzpicture} \begin{tikzpicture}[scale=0.55]
\draw [line width=0.25mm,fill=yellow] (A)--(B)--(C)--(D)--(E)--(F)--(A); 

\draw [line width=0.25mm,fill=green] (G)--(H)--(I)--(J)--(K)--(J1)--(K1)--(L1)--(E1)--(F1)--(A1)--(B1)--(C1)--(H1)--(G); 

\draw[dashed,opacity=0.4] (C1) -- (D1); 
\draw[dashed,opacity=0.4] (K) -- (L);  
\node[] (p) at ({sqrt(3)},0) {2};
\node[] (p) at ({2*sqrt(3)},0) {4};
\node[] (p) at ({sqrt(3)/2},-1.5) {1};
\node[] (p) at ({3*sqrt(3)/2},-1.5) {3};
\end{tikzpicture} \begin{tikzpicture}[scale=0.55]

\draw [line width=0.25mm,fill=yellow] (G1)--(H1)--(I1)--(J1)--(K1)--(L1)--(G1) ; 

\draw [line width=0.25mm,line width=0.25mm,fill=green] (A)--(B)--(C)--(H)--(I)--(J)--(K)--(L)--(G1)--(L1)--(E1)--(F1)--(A1)--(B1)--(A); 

\draw[dashed,opacity=0.4] (D) -- (E); 
\draw[dashed,opacity=0.4] (E) -- (F);  
 \draw[dashed,opacity=0.4] (L) -- (G);  
 
\node[] (p) at ({sqrt(3)},0) {2};
\node[] (p) at ({2*sqrt(3)},0) {4};
\node[] (p) at ({sqrt(3)/2},-1.5) {1};
\node[] (p) at ({3*sqrt(3)/2},-1.5) {3};
\end{tikzpicture} \begin{tikzpicture}[scale=0.55]
\draw [line width=0.25mm,fill=yellow] (A1)--(B1)--(C1)--(D1)--(E1)--(F1)--(A1); 
\draw [line width=0.25mm,fill=green] (A)--(B)--(C)--(H)--(I)--(J)--(K)--(J1)--(K1)--(L1)--(G1)--(G1)--(F)--(A); 
\draw[dashed,opacity=0.4] (C)--(D) -- (E); 
\draw[dashed,opacity=0.4] (K) -- (L);  
\node[] (p) at ({sqrt(3)},0) {2};
\node[] (p) at ({2*sqrt(3)},0) {4};
\node[] (p) at ({sqrt(3)/2},-1.5) {1};
\node[] (p) at ({3*sqrt(3)/2},-1.5) {3};
\end{tikzpicture}  \begin{tikzpicture}[scale=0.55]

\draw [line width=0.25mm,fill=green] (A)--(B)--(C)--(H)--(I)--(J)--(K)--(L)--(E)--(F)--(A); 
\draw [line width=0.25mm,fill=yellow] (A1)--(B1)--(C1)--(H1)--(I1)--(J1)--(K1)--(L1)--(E1)--(F1)--(A1); 

\draw[dashed,opacity=0.4] (C) -- (D); 
\draw[dashed,opacity=0.4] (C1) -- (D1);  
  
\node[] (p) at ({sqrt(3)},0) {2};
\node[] (p) at ({2*sqrt(3)},0) {4};
\node[] (p) at ({sqrt(3)/2},-1.5) {1};
\node[] (p) at ({3*sqrt(3)/2},-1.5) {3};
\end{tikzpicture} \begin{tikzpicture}[scale=0.55]

\draw [line width=0.25mm,fill=green] (A)--(B)--(C)--(D)--(E)--(D1)--(E1)--(F1)--(A1)--(B1)--(A); 
\draw [line width=0.25mm,fill=yellow] (G)--(H)--(I)--(J)--(K)--(J1)--(K1)--(L1)--(G1)--(H1)--(G); 

\draw[dashed,opacity=0.4] (F) -- (E); 
\draw[dashed,opacity=0.4] (K) -- (L);  
  
\node[] (p) at ({sqrt(3)},0) {2};
\node[] (p) at ({2*sqrt(3)},0) {4};
\node[] (p) at ({sqrt(3)/2},-1.5) {1};
\node[] (p) at ({3*sqrt(3)/2},-1.5) {3};
\end{tikzpicture} 

\begin{tikzpicture}[scale=0.55]
\draw [line width=0.25mm,fill=yellow] (A1)--(B1)--(C1)--(D1)--(E1)--(F1)--(A1); 
\draw [line width=0.25mm,fill=orange] (G1)--(H1)--(I1)--(J1)--(K1)--(L1)--(G1) ; 
\draw [line width=0.25mm,fill=green] (A)--(B)--(C)--(H)--(I)--(J)--(K)--(L)--(E)--(F)--(A); 
\draw[dashed,opacity=0.4] (D) -- (C); 
\node[] (p) at ({sqrt(3)},0) {2};
\node[] (p) at ({2*sqrt(3)},0) {4};
\node[] (p) at ({sqrt(3)/2},-1.5) {1};
\node[] (p) at ({3*sqrt(3)/2},-1.5) {3};
\end{tikzpicture} \begin{tikzpicture}[scale=0.55]

\draw [line width=0.25mm,fill=yellow] (A) -- (B) -- (C) -- (D) -- (E) -- (F) --(A); 
\draw [line width=0.25mm,fill=orange] (A1)--(B1)--(C1)--(D1)--(E1)--(F1)--(A1); 

\draw [line width=0.25mm,fill=green] (G)--(H)--(I)--(J)--(K)--(J1)--(K1)--(L1)--(G1)--(H1)--(G); 

\draw[dashed,opacity=0.4] (K) -- (L); 
  
\node[] (p) at ({sqrt(3)},0) {2};
\node[] (p) at ({2*sqrt(3)},0) {4};
\node[] (p) at ({sqrt(3)/2},-1.5) {1};
\node[] (p) at ({3*sqrt(3)/2},-1.5) {3};
\end{tikzpicture} \begin{tikzpicture}[scale=0.55]

\draw [line width=0.25mm,fill=green] (A) -- (B) -- (C) -- (D) -- (I1) -- (J1) --(K1)--(L1)--(E)--(F)--(A); 
\draw [line width=0.25mm,fill=orange] (G)--(H)--(I)--(J)--(K)--(L)--(G) ; 

\draw [line width=0.25mm,fill=yellow] (A1)--(B1)--(C1)--(D1)--(E1)--(F1)--(A1); 

\draw[dashed,opacity=0.4] (D) -- (E); 
  
\node[] (p) at ({sqrt(3)},0) {2};
\node[] (p) at ({2*sqrt(3)},0) {4};
\node[] (p) at ({sqrt(3)/2},-1.5) {1};
\node[] (p) at ({3*sqrt(3)/2},-1.5) {3};
\end{tikzpicture} \begin{tikzpicture}[scale=0.55]

\draw [line width=0.25mm,fill=yellow] (G)--(H)--(I)--(J)--(K)--(L)--(G) ; 
\draw [line width=0.25mm,fill=orange] (G1)--(H1)--(I1)--(J1)--(K1)--(L1)--(G1) ; 

\draw [line width=0.25mm,fill=green] (A)--(B)--(C)--(D)--(C1)--(D1)--(E1)--(F1)--(A1)--(B1)--(A); 

\draw[dashed,opacity=0.4] (F) -- (E); 
 
\node[] (p) at ({sqrt(3)},0) {2};
\node[] (p) at ({2*sqrt(3)},0) {4};
\node[] (p) at ({sqrt(3)/2},-1.5) {1};
\node[] (p) at ({3*sqrt(3)/2},-1.5) {3};
\end{tikzpicture}    \begin{tikzpicture}[scale=0.55]

\draw [line width=0.25mm,fill=yellow] (A) -- (B) -- (C) -- (D) -- (E) -- (F) --(A); 
\draw [line width=0.25mm,fill=orange] (G)--(H)--(I)--(J)--(K)--(L)--(G) ; 

\draw [line width=0.25mm,fill=green] (A1)--(B1)--(C1)--(H1)--(I1)--(J1)--(K1)--(L1)--(E1)--(F1)--(A1); 

\draw[dashed,opacity=0.4] (C1) -- (D1); 
  
\node[] (p) at ({sqrt(3)},0) {2};
\node[] (p) at ({2*sqrt(3)},0) {4};
\node[] (p) at ({sqrt(3)/2},-1.5) {1};
\node[] (p) at ({3*sqrt(3)/2},-1.5) {3};
\end{tikzpicture} 

\begin{tikzpicture}[scale=0.55]

\draw [line width=0.25mm,fill=green] (A) -- (B) -- (C) -- (D) -- (E) -- (F) --(A); 
\draw [line width=0.25mm,fill=yellow] (A1)--(B1)--(C1)--(D1)--(E1)--(F1)--(A1); 
\draw [line width=0.25mm,fill=orange] (G)--(H)--(I)--(J)--(K)--(L)--(G) ; 
\draw [line width=0.25mm,fill=cyan] (G1)--(H1)--(I1)--(J1)--(K1)--(L1)--(G1) ; 

\node[] (p) at ({sqrt(3)},0) {2};
\node[] (p) at ({2*sqrt(3)},0) {4};
\node[] (p) at ({sqrt(3)/2},-1.5) {1};
\node[] (p) at ({3*sqrt(3)/2},-1.5) {3};
\end{tikzpicture} 

\caption{All possible divisions of the strip with $4$ hexagons.} \label{fig:grid n=4}
\end{figure}
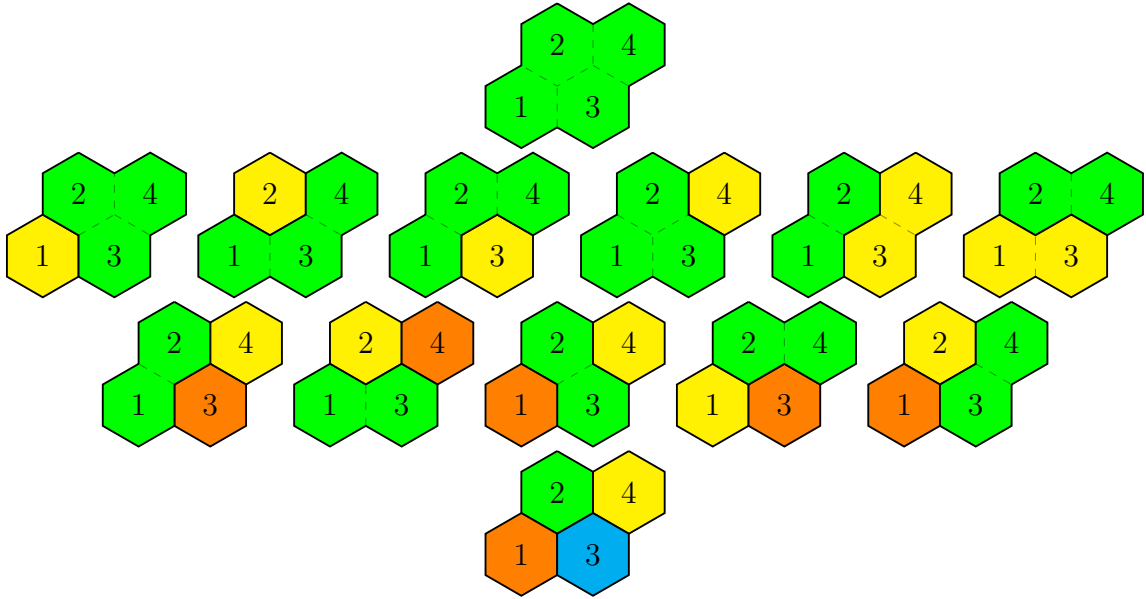 

Since the inner dual of $2\times n$ rectangular strip is a subgraph of the
inner dual of hexagonal strip of length $2n$, all divisions of a $2\times n$
rectangular strip are also valid divisions of a hexagonal strip with $n$
hexagons, but not vice versa. Figure \ref{fig:valid and not valid division n=4}
shows the division $\left\lbrace 1,23,4\right\rbrace$ which is legal in the
hexagonal strip but illegal in the rectangular strip.
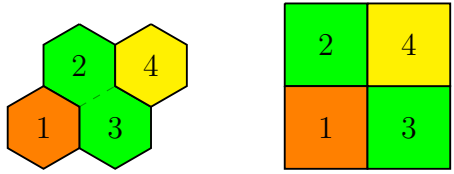
\begin{figure}[ht]
\centering \begin{tikzpicture}[scale=0.55]

\draw [line width=0.25mm,fill=green] (A) -- (B) -- (C) -- (D) -- (I1) -- (J1) --(K1)--(L1)--(E)--(F)--(A); 
\draw [line width=0.25mm,fill=orange] (G)--(H)--(I)--(J)--(K)--(L)--(G) ; 

\draw [line width=0.25mm,fill=yellow] (A1)--(B1)--(C1)--(D1)--(E1)--(F1)--(A1); 

\draw[dashed,opacity=0.4] (D) -- (E); 
  
\node[] (p) at ({sqrt(3)},0) {$2$};
\node[] (p) at ({2*sqrt(3)},0) {$4$};
\node[] (p) at ({sqrt(3)/2},-1.5) {$1$};
\node[] (p) at ({3*sqrt(3)/2},-1.5) {$3$};
\end{tikzpicture} \hspace{1 cm} \begin{tikzpicture}[scale=0.55]

\draw [line width=0.25mm,fill=green] (M1)--(M2)--(N2)--(N1)--(M1); 
\draw [line width=0.25mm,fill=orange] (N1)--(N2)--(P2)--(P1)--(N1); 
\draw [line width=0.25mm,fill=green] (N2)--(N3)--(P3)--(P2)--(N2); 
\draw [line width=0.25mm,fill=yellow] (M2)--(M3)--(N3)--(N2)--(M2); 

\node[] (p) at (1,-1) {$1$};
\node[] (p) at (1,1) {$2$};
\node[] (p) at (3,-1) {$3$};
\node[] (p) at (3,1) {$4$};
\end{tikzpicture} 
\caption{On the left is a valid division of a honeycomb strip, and on the right corresponding division of rectangular grid that is not allowed.} \label{fig:valid and not valid division n=4}
\end{figure}

\section{Dividing a chocolate bar into a given number of parts}

Numbers $r_k(n)$ of (\ref{eq:r_k(n)}) form a triangular array; its first few
lines are shown in Table \ref{tab:r_k(n)}.
In this section we investigate behavior of its columns, i.e., we turn our
attention to recursive relation for $r_k(n)$ where $k$ is fixed.
\let\mycolwd\relax 
\newlength{\mycolwd}                                        
\settowidth{\mycolwd}{$\frac{5444}{(k+1)}$}
\begin{table}[ht]
\centering
$\begin{array}[valign=c]{c| *{12}{@{}C{\mycolwd}@{}}}
n \backslash k& 1 & 2 & 3 & 4 & 5 & 6 & 7 & 8 & 9 & 10    \\[1.5ex]
\hline
1 &1 & 1\\
2 &1 & 6 & 4 & 1\\
3 &1 & 15 & 29 & 21 & 7 &1\\
4 &1 & 28 & 107 & 153 & 111 & 45 & 10 & 1\\
5 &1 & 45 & 286 & 678 & 831 & 603 & 274 & 78 & 13 & 1\\
\end{array}$
\caption{First few rows of of $r_k(n)$.}\label{tab:r_k(n)}
\end{table} 

As mentioned before, $r_0(n) = 0$ and $r_{1}(n)=1$, so we look at the first
non-trivial case, $k=2$. From relation (\ref{eq:r_k(n)}) we obtain
$r_{2}(n)=r_0(n-1)+3r_{1}(n-1)+2r_{2}(n-1)+r_{0}(n-2)+r_{1}(n-2)-r_{2}(n-2)$.
By plugging in $r_{0}(n)=0$ and $r_{1}(n)=1$, we obtain
\begin{align*}
r_{2}(n)=&2r_{2}(n-1)-r_{2}(n-2)+4\\
r_{2}(n-1)=&2r_{2}(n-2)-r_{2}(n-3)+4,\\
\end{align*}
and by subtracting these two equations we arrive at 
\begin{equation}
r_{2}(n)=3r_{2}(n-1)-3r_{2}(n-2)+r_{2}(n-3). \label{eq:r_2(n)}
\end{equation} 

Rewriting the trivial case 
\begin{equation}
r_1(n)=r_1(n-1)\label{eq:r_1(n)}
\end{equation} 
as
\begin{equation}
\binom{1}{0}r_{1}(n)=\binom{1}{1}r_{2}(n-1),
\end{equation}
and case $k = 2$ as
\begin{equation}
\binom{3}{0}r_{2}(n)=\binom{3}{1}r_{2}(n)-\binom{3}{2}r_{2}(n-1)+\binom{3}{3}r_{2}(n-2)
\end{equation}
suggests that there is a pattern valid also for higher values of $k$. The
conjectured pattern is readily verified by induction, thus yielding the
following theorem.
\begin{theorem}\label{tm:binom recursion r_k(n)}
For integers $n,k\geq 1$ we have
\begin{equation}
\sum\limits_{j=0}^{2k-1}(-1)^{j}\binom{2k-1}{j}r_{k}(n-j)=0.
\end{equation}
\end{theorem}
\begin{proof}
The proof is by induction. For $k=1,2$, the base of induction is true, as 
stated above. To verify the step of induction, we use recursion
(\ref{eq:r_k(n)}) to obtain a system of $2k-2$ equations as follows:
\begin{align*}
r_{k}(n)=\, &r_{k-2}(n-1)+3r_{k-1}(n-1)+2r_{k}(n-1)+\\&+r_{k-2}(n-2)+r_{k-1}(n-2)-r_{k}(n-2)\\
r_{k}(n-1)=\, &r_{k-2}(n-2)+3r_{k-1}(n-2)+2r_{k}(n-2)+\\&+r_{k-2}(n-3)+r_{k-1}(n-3)-r_{k}(n-3)\\
r_{k}(n-2)=\, &r_{k-2}(n-3)+3r_{k-1}(n-3)+2r_{k}(n-3)+\\&+r_{k-2}(n-4)+r_{k-1}(n-4)-r_{k}(n-4)\\
\vdots&\\
r_{k}(n-2k+3)=\, &r_{k-2}(n-2k+2)+3r_{k-1}(n-2k+2)+2r_{k}(n-2k+2)+\\&+r_{k-2}(n-2k+1)+r_{k-1}(n-2k+1)-r_{k}(n-2k+1)
\end{align*}
The term $r_{k}(n-j)$ appears in at most three equations, namely in
the $(j-1)^{\textup{st}}$, $j^{\textup{th}}$ and $(j+1)^{\textup{st}}$ equation.
To proceed forward, we multiply $j$-th equation by $(-1)^{j}\binom{2k-3}{j-1}$
and we add up all equations. For even $j$, the term $r_{k}(n-j)$ appears with
the coefficient
$$\binom{2k-3}{j-2}+2\binom{2k-3}{j-1}+\binom{2k-3}{j}=\binom{2k-1}{j},$$
and for odd $j$ with the same coefficient, but with the opposite sign. We
conclude that
\begin{equation}
r_{k}(n)=\sum\limits_{j=1}^{2k-1}(-1)^{j}\binom{2k-1}{j}r_{k}(n-j)+A_{k-1}(n)+A_{k-2}(n),
\end{equation}
where $A_{k-1}(n)$ and $A_{k-2}(n)$ are some expressions involving
$r_{k-1}(n-j)$ and $r_{k-2}(n-j)$, respectively. The claim of the Theorem
will be established if we show that both $A_{k-1}(n)$ and $A_{k-2}(n)$ are 
equal to zero. We first look at $A_{k-1}(n)$.
For $j\geq 1$, the term $r_{k-1}(n-j)$ appears twice in our system of
equations, in the $(j-1)^{\textup{st}}$ and in the $j^{\textup{th}}$ equation,
hence, the coefficient by $r_{k-1}(n-j)$ is
$\binom{2k-3}{j-2}-3\binom{2k-3}{j-1}$ for an odd $j$, and
$3\binom{2k-3}{j-1}-\binom{2k-3}{j-2}$ for an even $j$. So,
$$A_{k-1}(n)=\sum\limits_{j=1}^{2k-1}(-1)^{j}\left(3\binom{2k-3}{j-1}-\binom{2k-3}{j-2}\right)r_{k-1}(n-j).$$

For $k-1$ we can use the induction hypothesis, hence
 $$3\sum\limits_{j=0}^{2k-3}(-1)^{j}\binom{2k-3}{j}r_{k-1}(n-j-1)=0$$
and $$\sum\limits_{j=0}^{2k-3}(-1)^{j}\binom{2k-3}{j}r_{k-1}(n-j-2)=0.$$
After adding the equations we have
\begin{align*}
0&=3\sum\limits_{j=0}^{2k-3}(-1)^{j}\binom{2k-3}{j}r_{k}(n-j-1)+\sum\limits_{j=0}^{2k-3}(-1)^{j}\binom{2k-3}{j}r_{k}(n-j-2)\\
&=3\sum\limits_{j=1}^{2k-2}(-1)^{j-1}\binom{2k-3}{j-1}r_{k}(n-j)+\sum\limits_{j=2}^{2k-1}(-1)^{j}\binom{2k-3}{j-2}r_{k}(n-j)\\
&=\sum\limits_{j=1}^{2k-1}(-1)^{j}\left(\binom{2k-3}{j-2}-3\binom{2k-3}{j-1}\right) r_{k}(n-j),
\end{align*}
hence, $A_{k-1}(n)=0$. 

Similarly, $A_{k-2}(n)$ can be expressed as 
$$A_{k-2}(n)=\sum\limits_{j=1}^{2k-1}(-1)^{j}\left(\binom{2k-3}{j-1}-\binom{2k-3}{j-2}\right)r_{k-2}(n-j),$$
and, again, by using the induction hypothesis for $k-2$ we obtain 
$A_{k-2}(n)=0$. The proof follows along the same lines as for $A_{k-1}(n)=0$
and we omit the details.
This completes our proof. \end{proof}

Theorem \ref{tm:binom recursion r_k(n)} implies that all columns of the
array $r_k(n)$ are polynomials in $n$. Moreover, $r_k(n)$ is a polynomial
in $n$ of degree $2k-2$. The exact expressions can be easily 
obtained by fitting to the initial values, but we omit the details. Our
Theorem \ref{tm:binom recursion r_k(n)} reestablishes the polynomiality
results of references \cite{brown} and \cite{durham} in a more compact
and self-contained form.

A similar reasoning could be also employed near the upper end of the range
of $k$ and used to establish polynomiality of diagonals $r_{2n-k}(n)$,
going thus beyond the results of references \cite{brown,durham}.
Indeed, $r_{2n}(n) = 1$ for all non-negative integers $n$.
Furthermore, $r_{2n-1}(n) = 3n-2$, since among the $2n-1$ pieces there must
be exactly one dimer. That dimer is an edge in the inner dual of our bar,
hence an edge in a ladder graph with $n$ rungs, and there are exactly 
$3n-2$ such edges. In a similar way one can see that $r_{2n-2}(n)$ must be
a quadratic polynomial in $n$: A division into $2n-2$ parts can either
contain one trimer and $2n-3$ monomers, or two dimer and $2n-4$ monomers.
As the number of trimers is linear in $n$ and the number of pairs of dimers
is quadratic in $3n$, by fitting on the first few values for small $n$ one
obtains $r_{2n-2}(n) = \frac{9}{2}(n-1)(n-\frac{2}{3})$. By continuing
with the same reasoning, one obtains a general result.

\begin{theorem}\label{r_{n-k}(n)}
$r_{2n-k}(n)$ is a polynomial of degree $k$ in $n$ with the leading coefficient
$\frac{3^k}{k!}$.
\end{theorem}
We leave the details to the interested reader. 

We now move towards computing the bivariate generating function for $r_k(n)$.
Let $$F(x,y)=\sum\limits_{n\geq 1} \sum\limits_{k\geq 1}r_k(n)x^ny^k$$ 
denote the desired generating function. By starting from recurrence
(\ref{eq:r_k(n)}) we readily obtain
$$F(x,y)=\dfrac{xy(1-x+y+xy)}{1-(2+3y+y^2)x-(y^2+y-1)x^2}.$$ 
 
By substituting $y=1$ we obtain $$F(x,1)=\dfrac{2x}{1-6x-x^2},$$
the univariate generating function for the sequence $r_n$.

Now we can determine the expected number of pieces in a random division. We
rely on the following version of the Darboux's theorem \cite{bender}.

\begin{theorem}[Darboux]\label{tm:darboux} If the generating function
$f(x)=\sum_{n\geq 0}a_x^n$ of a sequence $(a_n)$ can we written in the form
$f(x)=\left(1-\frac{x}{\omega}\right)^{\alpha}h(x)$, where $\omega$ is the
smallest modulus singularity of $f$ and $h$ is analytic in $\omega$,
then $a_n\sim\frac{h(\omega)n^{-\alpha-1}}{\Gamma(-\alpha)\omega^n}$,
where $\Gamma$ denotes the gamma function.\end{theorem}

Since $\omega=\sqrt{10}-3$ we can write
$$F(x,1)=\dfrac{2x}{x\left(\sqrt{10}-3\right)+1}\left(1-\dfrac{x}{\left(\sqrt{10}-3\right)}\right)^{-1}.$$ Hence, we have
$h(x)=\dfrac{2x}{x\left(\sqrt{10}-3\right)+1}$ and
$h(\omega)=\frac{\sqrt{10}}{10}$.
Furthermore,
$$\left.\frac{\partial F(x,y)}{\partial y}\right\rvert_{y=1}=\frac{-x(x^3+3x^2+7x-3)}{\left(x\left(\sqrt{10}-3\right)x+1\right)^2}\left(1-\dfrac{x}{\left(\sqrt{10}-3\right)}\right)^{-2}$$
yields
$g(x)=\frac{-x(x^3+3x^2+7x-3)}{\left(x\left(\sqrt{10}-3\right)x+1\right)^2}$
and $g(\omega)=\frac{3\sqrt{10}-4}{20}$.
By Theorem \ref{tm:darboux}, the expected number of parts is
$$\dfrac{\dfrac{g(\omega)n}{\Gamma(2)\omega^n}}{\dfrac{h(\omega)}{\Gamma(1)\omega^n}}=\left(\dfrac{3}{2}-\sqrt{\frac{2}{5}}\right) n.$$
Hence, we have established the following result for the expected number of
parts in a random division of a chocolate bar of length $n$.
\begin{theorem}
The expected number of parts in a random division of a chocolate bar of
length $n$ is given by $$\left(\dfrac{3}{2}-\sqrt{\frac{2}{5}}\right) n
\approx 0.867544 \,n.$$
\end{theorem}
The above result is derived under the so-called equilibrium assumption, where
all divisions are equally likely.

The triangle of Table \ref{tab:r_k(n)} is not (yet) in the OEIS \cite{oeis}.
However,
its row sums appear as A078469, the number of compositions of ladder graphs
in the sense of reference \cite{knopf}. Hence our results could be also
interpreted as a refinement of the number of compositions of ladder graphs.
Sequence $r_3(n)$ appears as A345897, with the same interpretation as we
give here. Curiously, such an interpretation seems to be missing among
many combinatorial interpretations of A000384, the hexagonal numbers, which
appear as the second column of our triangle. Similarly, $r_{2n-2}(n)$
appears as A081266, but without the interpretation given here.

\section{Divisions of honeycomb strips}

\subsection{Recurrences, explicit formulas and generating functions}

Recall that $D_k(n)$ denotes the set of all possible divisions of the
honeycomb strip with $n$ hexagons into $k$ pieces and $d_k(n)=|D_k(n)|$
the number of elements of the set $D_k(n)$. In order to count the divisions
correctly, special attention must be paid to the rightmost two cells, since
the new $(n+1)^{\textup{st}}$ cell can interact only with them. Whether these
hexagons are in the same pieces or not plays crucial role in how the  new
hexagon can be added. We denote by $S_k(n)$ the set of all possible divisions
of the honeycomb strip with $n$ hexagons into $k$ pieces with two last 
hexagons in the different parts. Similarly, let $T_k(n)$ denote the set of
all possible divisions of the strip into $k$ pieces with two two rightmost
hexagons belonging to the same piece. Let $s_k(n)=|S_k(n)|$ and
$t_k(n)=|T_k(n)|$.  Since the last two hexagons can either be together or
separated, we have divided the set $D_k(n)$ into two disjoint sets,
$D_k(n)=S_k(n)\cup T_k(n)$, hence $d_k(n)=t_k(n)+s_k(n)$.  

\begin{figure}[ht]
\centering \begin{tikzpicture}[scale=0.55]

\draw [line width=0.25mm] (F)--(A) (K)--(J); 
\draw [line width=0.25mm, dotted] (A)--(B) (J)--(I); 

\draw [line width=0.25mm,fill=green] (A1)--(B1)--(C1)--(H1)--(I1)--(J1)--(K1)--(L1)--(E1)--(F1)--(A1); 

\draw[dashed,opacity=0.4] (C1)--(D1); 
\end{tikzpicture}

\caption{A honeycomb strip with two rightmost hexagons in the same piece.} 
\end{figure}
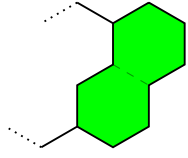 
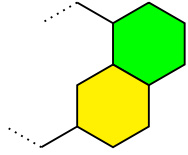
\begin{figure}[ht]
\centering
\begin{tikzpicture}[scale=0.55]

\draw [line width=0.25mm,fill=green] (A1)--(B1)--(C1)--(D1)--(E1)--(F1)--(A1); 
\draw [line width=0.25mm,fill=yellow] (G1)--(H1)--(I1)--(J1)--(K1)--(L1)--(G1) ; 

\draw [line width=0.25mm] (F)--(A) (K)--(J); 
\draw [line width=0.25mm, dotted] (A)--(B) (J)--(I); 

\draw[dashed,opacity=0.4] (G1) -- (H1); 
\end{tikzpicture} 
\caption{A honeycomb strip with two rightmost hexagons in different pieces.} 
\end{figure} 

We first establish an auxiliary result.

\begin{theorem}\label{tm:recursive s(n+1,k)} For $n\geq 1$, the number of all
possible divisions $s_k(n)$ of the honeycomb strip with $n$ hexagons into
$k$ pieces, with hexagons in the last column being in the different pieces,
satisfies the following relation: 
\begin{align}\label{eq:s(n,k) recursive} s_k(n+1)&=s_{k-1}(n)+2s_k(n)-s_{k}(n-1).\end{align}
\end{theorem}

\begin{proof} We start with a strip containing $n$ hexagons and add one new
hexagon to obtain a strip with $n+1$ hexagons. The new  hexagon can either
increase the number of parts in division by $1$ or not increase this number.
To obtain a division with $k$ pieces, we can only start with the
division with $k-1$ or $k$ pieces. These are two disjoint sets, so the
number of all divisions will be the sum of these cases.

When starting with division consisting of $k-1$ pieces, we can obtain $k$
pieces by adding the new hexagons as individual pieces. Since there is only
one way to do that, the number of divisions that can be obtained this way is
$d_{k-1}(n)$. Note that the condition that rightmost two hexagons belong to
different pieces is satisfied, as shown in
Figure \ref{Fig:S(n+1,k) from D(n,k-1)}. 

\begin{figure}[ht]
\centering \begin{tikzpicture}[scale=0.55]

\draw [line width=0.25mm,fill=green] (A1)--(B1)--(C1)--(D1)--(E1)--(F1)--(A1); 
\draw [line width=0.25mm,fill=yellow] (G1)--(H1)--(I1)--(J1)--(K1)--(L1)--(G1) ; 

\draw [line width=0.25mm] (F)--(A) (K)--(J); 
\draw [line width=0.25mm, dotted] (A)--(B) (J)--(I);


\draw[dashed,opacity=0.4] (G1) -- (H1); 
  
\node[font=\small] (p) at ({2*sqrt(3)},0) {$n+1$};
\node[font=\small] (p) at ({3*sqrt(3)/2},-1.5) {$n$};
\end{tikzpicture} 
\caption{The element of $S_{k}(n+1)$ obtained from the element of
$D_{k-1}(n)$ by adding the new hexagon as separated piece. }
\label{Fig:S(n+1,k) from D(n,k-1)}
\end{figure}
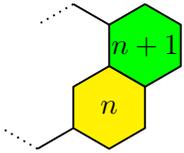

It remains to consider one last case. We start with a strip divided into $k$
pieces and we add $(n+1)^{\textup{st}}$ hexagon. If the last two hexagons
in the division are together, we cannot add new hexagons so that
number of parts remains the same and that new hexagons are in different
pieces. 

Now we move to the case where the last two hexagons in the division
are separated. There is only one way to add new hexagons to the existing
strip, to put the $(n+1)^{\textup{st}}$ hexagon together with
$(n-1)^{\textup{st}}$ (see Figure \ref{Fig:S(n+1,k) from S(n,k)}). Every
other layout would be in contradiction with either number of pieces or the
fact that two last hexagons should be separated, since putting
$(n+1)^{\textup{st}}$ hexagon together with $n^{\textup{th}}$ hexagon would
produce the element of $T_k(n)$. So in this case we have $s_k(n)$ ways to
obtain the desired division.

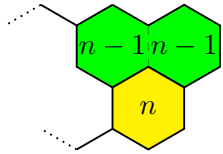
\begin{figure}[ht]
\centering \begin{tikzpicture}[scale=0.55]

\draw [line width=0.25mm,fill=green] (A1)--(B1)--(C1)--(D1)--(E1)--(D2)--(E2)--(F2)--(A2)--(B2)--(A1); 
\draw [line width=0.25mm,fill=yellow] (G2)--(H2)--(I2)--(J2)--(K2)--(L2)--(G2) ; 

\draw [line width=0.25mm] (F)--(A) (K1)--(J1); 
\draw [line width=0.25mm, dotted] (A)--(B) (J1)--(I1); 

\draw[dashed,opacity=0.4] (E1) -- (F1);

\node[font=\small] (p) at ({2*sqrt(3)},0) {$n-1$};
\node[font=\small] (p) at ({3*sqrt(3)},0) {$n-1$};
\node[font=\small] (p) at ({5*sqrt(3)/2},-1.5) {$n$};
\end{tikzpicture} 
\caption{The element of $S_k(n+1)$ obtained from the element of $S_k(n)$ by
joining the new hexagon with $(n-1)^{\textup{st}}$ hexagon}
\label{Fig:S(n+1,k) from S(n,k)}
\end{figure} 

By summing these two cases, we obtain the recursive relation
\begin{equation} \label{eq:s(n,k) aux}
s_k(n+1)=d_{k-1}(n)+s_{k}(n).
\end{equation} 

To eliminate $d_{k-1}(n)$ from relation (\ref{eq:s(n,k) aux}), we use the
fact that $d_{k-1}(n)=t_{k-1}(n)+s_{k-1}(n)$. By removing the last hexagon
from the strip, we establish a 1-to-1 correspondence between all division of a
strip with $n-1$ hexagons and divisions of a strip with $n$ hexagons where
two last hexagons are in the same part. Hence, $t_k(n)=d_k(n-1)$. Then we
have
\begin{align*}
s_{k}(n+1)&=s_{k-1}(n)+t_{k-1}(n)+s_k(n)\\
&=s_{k-1}(n)+d_{k-1}(n-1)+s_k(n), 
\end{align*} 
hence $d_{k-1}(n-1)=s_k(n+1)-s_{k-1}(n)-s_k(n)$, which combined with
relation (\ref{eq:s(n,k) aux}) yields
\begin{align*}
s_k(n+1)&=s_{k-1}(n)+2s_k(n)-s_{k}(n-1)
\end{align*}
and we proved the theorem.
\end{proof}  
  
By disregarding values of $k$ in recursive relation \ref{eq:s(n,k) recursive}
we obtain $$s(n+1)=3s(n)-s(n-1),$$ where $s(n)$ represents the number of
all divisions of a honeycomb strip of length $n$ with two last hexagons in
different parts. Since we obtained the same recursive relation as for bisection
of Fibonacci sequence with $s(1)=0$ and $s(2)=1$ we have $$s(n)=F_{2n-2}.$$  

Our main result of this section now follows by much the same reasoning, as
$d_k(n)$ satisfy the same recurrence as $s_k(n)$. We state it without
proof.
\begin{theorem}\label{tm:recursive d(n+1,k)} For $n\geq 1$, the number of
all possible divisions $d_k(n)$ of $n$  honeycomb strip into $k$ pieces
satisfies the following relation: 
\begin{align}\label{eq:d(n,k) recursive} 
d_k(n+1)= d_{k-1}(n)+2d_k(n)-d_k(n-1).
\end{align}
\end{theorem}

Again, by grouping together terms of recurrence \ref{eq:d(n,k) recursive}
with respect to $n$, we obtain the recurrence satisfied by the sequence
$d(n)$ counting the total number of subdivisions of a honeycomb strip of
length $n$ as
$$ d(n+1) = 3d(n) - d(n-1).$$
Taking into account the initial conditions $d(1) = 1$ and $d(2) = 2$ yields
a very simple answer.
\begin{theorem}\label{odd_fibo}
The total number of divisions of a honeycomb strip of length $n$ is given
by $d(n) = F_{2n-1}$, where $F_n$ denotes the $n^{\rm th}$ Fibonacci
number.
\end{theorem}
The above theorem yields a nice combinatorial interpretation of the 
odd-indexed Fibonacci numbers which seems to be absent from the entry
A001519 in the OEIS. It would be interesting to obtain our Theorem 
\ref{odd_fibo} by establishing a bijection between our divisions and
some of the objects listed there.

With the above result at hand, it is not too difficult to guess the 
explicit formulas for $d_k(n)$ and $s_k(n)$. The following theorem is
easily proved by simply verifying that the proposed expressions satisfy
the respective recurrences and initial conditions, and we omit the
details. Again, it would be more interesting to verify the formulas
in a combinatorial way.
\begin{theorem} \label{tm:closed form d(n,k)} The number of all divisions
$d_k(n)$ of the honeycomb strip with $n$ hexagons into exactly $k$ pieces is
$$d_k(n)=\binom{n+k-2}{n-k}.$$
The number $s_k(n)$ of all divisions of the honeycomb strip with $n$
hexagons into $k$ pieces such that two rightmost hexagons belong to
different pieces is equal to zero if $n=1$ and for $n \geq 2$ it is given as
$$s_k(n)=\binom{n+k-3}{n-k}.$$
\end{theorem}

Even though sequences $r_k(n)$ and $d_k(n)$ satisfy
different recursive relations and describe different problems, it turns
out that their columns satisfy the same recurrences.
Our next theorem is analogue of Theorem \ref{tm:binom recursion r_k(n)},
but for the sequence $d_k(n)$. We state it without proof.
\begin{theorem}\label{tm:binom recursion d_k(n)}
For $n,k\geq 1$ we have
\begin{equation}
\sum\limits_{j=0}^{2k-1}(-1)^{j}\binom{2k-1}{j}d_{k}(n-j)=0.
\end{equation}
\end{theorem}

As with a rectangular strip, we are now interested in generating function of
sequence $d_k(n)$. Let
$$G(x,y)=\sum\limits_{n\geq 1}\sum\limits_{k\geq 1}d_k(n)x^ny^k.$$
By recursive relation $\ref{eq:d(n,k) recursive}$ we have
\begin{align*}
G(x,y)&=xy+x^2y\left(1+y\right)+ \sum\limits_{n\geq 3}\sum\limits_{k\geq 1}\left(d_{k-1}(n-1)+2d_k(n-1)-d_k(n-2)\right)x^ny^k \\
&=xy+x^2y\left(1+y\right)+xy\left(G(x,y)-xy\right) +2x \left(G(x,y)-xy\right)-x^2G(x,y),
\end{align*}
so we have
$$G(x,y)=\dfrac{xy(1+x(y-1)-xy)}{1-(2+y)x+x^2}.$$
By putting $y=1$ we obtain the univariate generating function for the sequence
$d(n)$ as
$$G(x,1)=\dfrac{x-x^2}{1-3x+x^2}.$$
Its smallest-modulus singularity is $\omega=\frac{1}{2}\left(3-\sqrt{5}\right)$
and this gives us the asymptotics of the expected number of pieces in random
divisions of honeycomb strips of a given length.

\begin{theorem}\label{exp_hex}
The expected number of pieces in a random division of a honeycomb strip of
length $n$ asymptotically behaves as
$$\frac{\sqrt{5}}{5}n \approx 0.447214 n.$$
\end{theorem}
The proof follows by a straightforward application of Darboux theorem and
we omit the details.

%

\subsection{Some consequences}

Our results make possible to give new combinatorial interpretation for
some famous identities. We present two such cases.

First, by double counting the set $D(n)$, we gave
new meaning to the well-known identity
$$\sum_{k=1}^n\binom{n+k-2}{n-k}=F_{2n-1}.$$ 
Another identity will be proven in the next theorem.
\begin{theorem}
For $n,m\geq 1$ we have $$F_{2n+2m-1}=F_{2n-1}F_{2m-1}+F_{2n}F_{2m}.$$
\end{theorem}
\begin{proof}
We start with two honeycomb strips of lengths $n$ and $m$. To prove the
statement of a theorem we glue strips together as in
Figure \ref{fig:glued strips} and double count the number of divisions.
On one hand, we have a strip of length $n+m$ whose number of divisions is
$d(n+m)$. On the other hand, we consider what can happen when strips are
glued together. In the first case parts of each division do not interact,
hence we have $d(n)d(m)$ such divisions. In the other cases, at least two
parts, one from each strip, must merge. But to correctly count the number
of divisions in those cases, it is important to know whether division is
with two last hexagons together or separated. If both strips have the last
two hexagons hexagons together, the total number of such divisions is
$t(n)t(m)$. If both strips have last two hexagons hexagons separated,
the total number of such divisions is $4s(n)s(m)$, since there are four
different ways to merge parts. Finally, if one strip has two last hexagons
separated and the other one together, we can merge parts in two ways.
Since either one of strips can be in both situations, the total number of
divisions in this case is $2s(n)t(m)+2s(n)t(m)$.
\begin{figure}[ht]
\centering

\begin{tikzpicture}[scale=0.55]

\draw [line width=0.25mm] (A)--(B)--(C)--(D)--(E)--(F)--(A); 
\draw [line width=0.25mm] (G)--(H)--(I)--(J)--(K)--(L)--(G); 

\draw [line width=0.25mm] (A2)--(B2)--(C2)--(D2)--(E2)--(F2)--(A2); 
\draw [line width=0.25mm] (A3)--(B3)--(C3)--(D3)--(E3)--(F3)--(A3); 
\draw [line width=0.25mm] (G2)--(H2)--(I2)--(J2)--(K2)--(L2)--(G2); 
\draw [line width=0.25mm] (G3)--(H3)--(I3)--(J3)--(K3)--(L3)--(G3); 

\draw [line width=0.25mm] (A5)--(B5)--(C5)--(D5)--(E5)--(F5)--(A5); 
\draw [line width=0.25mm] (G5)--(H5)--(I5)--(J5)--(K5)--(L5)--(G5);

\node[] (p) at ({sqrt(3)},0) {$2$};
\node[] (p) at ({3*sqrt(3)},0) {$n$};
\node[] (p) at ({4*sqrt(3)},0) {$m-1$};
\node[] (p) at ({6*sqrt(3)},0) {$1$};

\node[] (p) at ({sqrt(3)/2},-1.5) {$1$};
\node[] (p) at ({3.55*sqrt(3)/2},-0.75) {$\cdots$};
\node[] (p) at ({5*sqrt(3)/2},-1.5) {$n-1$};
\node[] (p) at ({7*sqrt(3)/2},-1.5) {$m$};
\node[] (p) at ({9.55*sqrt(3)/2},-0.75) {$\cdots$};
\node[] (p) at ({11*sqrt(3)/2},-1.5) {$2$};
\end{tikzpicture}  
\caption{Two honeycomb strips glued together.} \label{fig:glued strips}
\end{figure}
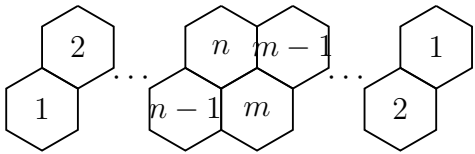 

So, \begin{align*}
d(n+m)&=d(n)d(m)+4s(n)s(m)+t(n)t(m)+2s(n)t(m)+2s(n)t(m).
\end{align*}
The claim now follows by substituting $d(n)=F_{2n-1}$, $s(n)=F_{2n-2}$ and
$t(n)=F_{2n-3}$ and rearranging the resulting expressions.
\end{proof}

The results of this section can be also formulated in terms of graph
compositions, this time of the graph $P_n^2$ obtained by adding edges 
between all pairs of vertices at distance $2$ in $P_n$, the path on $n$
vertices. The following results is a direct consequence of the fact that
$P_n^2$ is the inner dual of a honeycomb strip of length $n$.
\begin{theorem}
The number of compositions of $P_n^2$ with $k$ components is equal to
$\binom{n+k-2}{n-2}$. The total number of compositions of $P_n^2$ is equal
to $F_{2n-1}$.
\end{theorem}

\section{Transfer matrix method}

\subsection{Honeycomb strips}

In this section we present another approach to obtain overall number
of divisions, the one based on transfer matrices. It might seem less
natural than recurrence relations, but it often turns out to be suitable
when recurrence relations are complicated or unknown.

We again consider a honeycomb strip such as the one shown in the
Figure \ref{fig:grid1}, and look at its rightmost column, i.e., at the
hexagons labeled by $n-1$ and $n$. There are two possible situations 
regarding these hexagons: they can be in the same piece of a subdivision,
or they can belong to two different pieces. We denote a strip with last
two hexagons together as a type $T$ strip and a strip with last two
hexagons separated as a type $S$ strip. Adding the $(n+1)-$st hexagon might
result again in a type $S$ strip or in a type $T$ strip. There are 
altogether four possibilities, each of them producing certain effects
on the number of pieces in the resulting strip. For example, if we start
with a strip of type $S$ and we want to end with a strip of type $S$,
we can either add the new hexagon to the part which contains the $(n-1)^{\textup{st}}$
hexagon, or we can let the $(n+1)^{\textup{st}}$ hexagon to form its own part. In the
first case, the number of parts will remain the same, in the second case
it will increase by one.
Figure \ref{fig:columns} shows this case.
\begin{figure}[ht]
\centering \begin{tikzpicture}[scale=0.55]
\draw [line width=0.25mm,fill=green] (A1)--(B1)--(C1)--(D1)--(E1)--(F1)--(A1);
\draw [line width=0.25mm,fill=yellow] (G1)--(H1)--(I1)--(J1)--(K1)--(L1)--(G1);
\draw [line width=0.25mm,fill=orange] (G2)--(H2)--(I2)--(J2)--(K2)--(L2)--(G2);
\node[font=\small] (p) at ({2*sqrt(3)},0) {$n$};
\node[font=\small] (p) at ({5*sqrt(3)/2},-1.5) {$n+1$};
\node[font=\small] (p) at ({3*sqrt(3)/2},-1.5) {$n-1$};
\end{tikzpicture} \begin{tikzpicture}[scale=0.55]
\draw [line width=0.25mm,fill=green] (A1)--(B1)--(C1)--(D1)--(E1)--(F1)--(A1);
\draw [line width=0.25mm,fill=yellow] (G1)--(H1)--(I1)--(J1)--(K1)--(J2)--(K2)--(L2)--(G2)--(H2)--(G1);
\node[font=\small] (p) at ({2*sqrt(3)},0) {$n$};
\node[font=\small] (p) at ({5*sqrt(3)/2},-1.5) {$n+1$};
\node[font=\small] (p) at ({3*sqrt(3)/2},-1.5) {$n-1$};
\draw[dashed,opacity=0.4] (H2)--(I2);
\end{tikzpicture} \label{fig:columns}
\caption{Both cases resulting with a strip of type $S$.}
\end{figure}
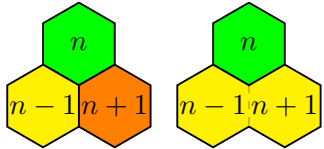
The main idea of the transfer matrix method is to
arrange the effects of adding a single hexagon into a $2 \times 2$ matrix 
whose entries will keep track of the number of pieces via a formal
variable, say, $y$. The rows and columns of such a matrix are indexed by
possible states, in our case $T$ and $S$, and the element at the position 
$S,S$ in our example will be $1+y$. That clearly captures the fact that
transfer from $S$ to $S$ results either in the same number of pieces, or
the number of pieces increases by one. The other three possible transitions,
$T \to T$, $S \to T$ and $T \to S$ are described by matrix elements
$1, 1$, and $y$, respectively. Indeed, it is clear that adding a hexagon
so to obtain the rightmost column together cannot increase the number of
pieces, hence the two ones, and that starting from $T$ and arriving at $S$
is possible only by the last hexagon forming a new piece, hence increasing
the number of pieces by one, hence $y$. If we denote our matrix by $H$,
we can write it as
\begin{equation*}
H(y)=\begin{bmatrix}
1 & 1 \\ y & 1+y 
\end{bmatrix}.
\end{equation*}
By construction, it is clear that adding a new hexagon will be well
described by multiplying some vector of states by our matrix $H(y)$, and that
repeated addition of hexagons will correspond to multiplication by powers
of $H(y)$. It remains to account for the initial conditions.
For the initial value $n=1$ we
have trivial case, one hexagon forms one part. For $n=2$ we have two
possibilities, hexagons are in the same part or separated. Hence his case
is represented by a vector
\begin{equation*}
\overrightarrow{h_2}=x\begin{bmatrix} y \\ y^2 \end{bmatrix}.
\end{equation*}
By introducing another formal variable, say $x$, to keep track of the length,
the above procedure will produce a sequence of bivariate polynomials whose
coefficients are our numbers $d_k(n)$. A first few polynomials are shown
in Table \ref{table:matrix} after the theorem which summarizes the described
procedure.
\begin{theorem} \label{tm:matrix}
The number of divisions of a honeycomb strip of a length $n$ into $k$ parts is the coefficient by $x^ny^k$ in the expression
\begin{equation}\label{expression}
\begin{bmatrix} 1& 1 \end{bmatrix} \begin{bmatrix}
1 & 1 \\ y & 1+y \end{bmatrix}^{n-2} \begin{bmatrix}
y \\ y^2 \end{bmatrix} x^n.
\end{equation}  
\end{theorem}

\begin{table}[ht]\centering \label{table:matrix}
\begin{tabular}{c|l}
$n$ & \\
\hline
$1$ & $xy$\\ 
$2$ & $x^2(y+y^2)$\\ 
$3$ & $x^3(y+3y^2+y^3)$\\ 
$4$ & $x^4(y+6y^2+5y^3+y^4)$\\ 
$5$ & $x^5(y+10y^2+15y^3+7y^4+y^5)$\\
$6$ & $x^6(y+15y^2+35y^3+28y^4+9y^5+y^6)$\\
\end{tabular} \caption{First few bivariate polynomials from the transfer
matrix method. }
\end{table}

The coefficients by $x^ny^k$ in expression (\ref{expression}) could be now
determined by studying powers of the transfer matrix. By looking at the
first few cases,
$$H(y)^2=\begin{bmatrix}
1+y & 2+y \\ 2y+y^2 & 1+3y+y^2\end{bmatrix} \quad {\rm and}\quad  H(y)^3=\begin{bmatrix}
1+3y+y^2 & 3+4y+y^2 \\ 3y+4y^2+y^3 & 1+6y+5y^2+y^3\end{bmatrix},$$
we could guess the entries in the general case and then verify them by
induction. We state the result omitting the details of the proof.
\begin{lemma} \label{lm:matrix xy}
Matrix $$H(y)^n=\begin{bmatrix}
p(n) & s(n) \\ ys(n) & p(n+1)
\end{bmatrix}$$ with $p(n)=\displaystyle\sum\limits_{k=1}^{n}\binom{n+k-2}{n-k}y^{k-1}$ and $s(n)=\displaystyle\sum\limits_{k=1}^{n}\binom{n+k-1}{n-k}y^{k-1}$.
\end{lemma}
%
%
%
Lemma \ref{lm:matrix xy} allows us to simplify the expression (\ref{expression}) to have
\begin{align*}
\begin{bmatrix} 1& 1 \end{bmatrix} \begin{bmatrix}
1 & 1 \\ y & 1+y \end{bmatrix}^{n-2} \begin{bmatrix}
y \\ y^2 \end{bmatrix} x^n &= \begin{bmatrix} 1& 1 \end{bmatrix} \begin{bmatrix}
p(n-2) & s(n-2) \\ ys(n-2) & p(n-1) \end{bmatrix} \begin{bmatrix}
y \\ y^2 \end{bmatrix} x^n \\
&= \left(yp(n-2) + ys(n-2) + y^2s(n-2) + y^2p(n-1)\right)  x^n \\
&= \left(yp(n-1) + y\left(s(n-2) + yp(n-1)\right)\right)  x^n \\
&= p(n)x^ny \\
&=\displaystyle\sum\limits_{k=1}^{n}\binom{n+k-2}{n-k}y^{k} x^n 
\end{align*} 

By Theorem \ref{tm:matrix} we have $$d(n,k)=\binom{n+k-2}{n-k}.$$

Now we turn our attention to the number of all possible divisions, i.e.
we wish to determine the number $d(n)$. To to that, we again use matrix
$H(y)$ and Theorem \ref{tm:matrix}. By setting $y=1$ we have
$H(1)=\begin{bmatrix}
1 & 1 \\ 1 & 2 \end{bmatrix}=\begin{bmatrix}
F_1 & F_2 \\ F_2 & F_3 \end{bmatrix}$. Again, the following claim is easily
guessed and verified by induction.
\begin{lemma} \label{lm:matrix} $H(1)^n=\begin{bmatrix}
F_{2n-1} & F_{2n} \\ F_{2n} & F_{2n+1}
\end{bmatrix}$. \end{lemma}

Finally, by Lemma \ref{lm:matrix} we can simplify expression
(\ref{expression}) to have 
\begin{align*}  \begin{bmatrix} 1& 1 \end{bmatrix} \begin{bmatrix} F_{2n-5} & F_{2n-4} \\ F_{2n-4} & F_{2n-3} \end{bmatrix} \begin{bmatrix} 1 \\ 1 \end{bmatrix} &= \begin{bmatrix} F_{2n-5}+F_{2n-4}& F_{2n-4}+F_{2n-3} \end{bmatrix}\begin{bmatrix} 1 \\ 1 \end{bmatrix} \\ 
&=\begin{bmatrix} F_{2n-3}& F_{2n-2}\end{bmatrix}\begin{bmatrix} 1 \\ 1 \end{bmatrix} \\
&= F_{2n-1}. \end{align*}

By Theorem \ref{tm:matrix} we have $d(n)=F_{2n-1}$.

\subsection{Chocolate bars}

Transfer matrices can be also used to obtain the sequence $r_{k}(n)$
denoting the number of ways to divide rectangular strip $2\times n$ into
$k$ parts. In this case we do not add square by square, but column by column.
So, let $T$ denote a division of a strip where squares in the last column
are in the same part and $S$ a division where squares in the last column are
in different parts. 

For $n=1$ we have the same case as $n=2$ in a honeycomb strip, so this case
is represented by a vector
\begin{equation*}
\overrightarrow{q_1}=x\begin{bmatrix} y \\ y^2 \end{bmatrix}.
\end{equation*}

Similar as in the honeycomb case, if we start with a division of type $T$
and we wish to obtain another division of type $T$, we can do that either by
appending two new squares to the same part with the squares of the last
column or we can let two new squares form a new part. Hence, the
corresponding entry in the transfer matrix is $1+y$. By doing similar
analysis for other cases, we obtain the transfer matrix
\begin{equation*}
Q(y)=\begin{bmatrix}
1+y & 2+y \\ y(2+y) & (1+y)^2 
\end{bmatrix}.
\end{equation*}
Again, $y$ is a formal variable keeping track of the number of pieces.
So, for a strip $2\times n$, the coefficient by $x^ny^k$ in the expression
$$\begin{bmatrix}
1 & 1 
\end{bmatrix}\begin{bmatrix}
1+y & 2+y \\ y(2+y) & (1+y)^2 
\end{bmatrix}^{n-1}\begin{bmatrix}
y \\ y^2
\end{bmatrix}x^n$$
represents the number of ways to divide rectangular strip $2\times n$ into
exactly $k$ parts.

We conclude this section by mentioning that in both cases we could have
obtained the asymptotic behavior of numbers $d(n)$ and $r(n)$ by computing
the leading eigenvalue of the corresponding transfer matrix.

\section{Concluding remarks}
\noindent
In this paper we have employed two different methods to count 
divisions of narrow strips of squares and hexagons, respectively, into a 
given number of pieces, when cutting is allowed only along the edges of
basic polygons. We have obtained several triangular integer arrays and 
determined formulas for their entries. Despite similar settings, the two
problems behave in different ways: for honeycomb strips the entries of the
enumerating triangles are given as binomial coefficients with parameters
dependent on the strip length and the number of pieces, while for chocolate
bars no closed-form expression has been obtained. We were able to show, though,
that the entries in columns satisfy
convolution-type recurrences with coefficients forming alternating rows of
Pascal triangle.

Both problems were then addressed by using the transfer-matrix formalism.
The original results for the total number of divisions were re-derived 
in a more compact way, demonstrating thus the power of transfer-matrix method.
However, we found the approach unsuitable for refining the aggregate results,
for establishing the polynomial nature of columns and for obtaining
closed-form solutions in the rectangular case. Nevertheless, we believe that 
the transfer matrices would prove useful in treating a number of similar
problems, as indicated by our experiments with wider strips in both square 
and hexagonal lattices and with narrow strips in the triangular lattice.

\section*{Acknowledgement}
Partial support of Slovenian ARRS (Grant no. J1-3002)
is gratefully acknowledged by T. Do\v{s}li\'c.

\nocite{*}
\bibliographystyle{amsplain}
\bibliography{}

\end{document}